\documentclass[leqno]{amsart} %

\usepackage{array} 
\usepackage{amsmath}
\usepackage{amscd}
\usepackage{amsfonts}
\usepackage{amssymb}
\usepackage{graphicx} 
\usepackage[mathscr]{eucal}
\usepackage[usenames,dvipsnames]{color}
\usepackage[all]{xy}
\newdir{ >}{{}*!/-10pt/\dir{>}}

\usepackage[colorlinks]{hyperref}

\newtheorem{theorem}[equation]{Theorem}
\newtheorem{lemma}[equation]{Lemma}
\newtheorem{proposition}[equation]{Proposition}
\newtheorem{corollary}[equation]{Corollary}

\theoremstyle{remark}
\newtheorem{definition}[equation]{Definition}
\newtheorem{remark}[equation]{Remark}

\newcommand{\nc}{\newcommand}
\nc{\dmo}{\DeclareMathOperator}

\dmo{\Arr}{Arr}
\dmo{\Ch}{Ch}
\dmo{\Der}{D}
\dmo{\Ho}{Ho}
\dmo{\Mod}{Mod}
\dmo{\op}{op}
\dmo{\Qcoh}{Qcoh}
\dmo{\Spec}{Spec}

\nc{\aka}{{a.\,k.\,a.\ }}
\nc{\eg}{\textsl{e.g.}}
\nc{\ie}{\textsl{i.e.\ }}

\nc{\cat}[1]{\mathscr #1}
\nc{\Cat}{\mathsf{Cat}}
\nc{\CAT}{\mathsf{CAT}}
\nc{\DER}{\bbD\mathrm{er}}
\nc{\inv}{^{-1}}
\nc{\isotoo}{\,\otoo{\cong}\,}
\nc{\HO}{\mathbb{HO}}
\nc{\Loop}{\underline{\bbN}}
\nc{\MMod}{\text{-}\kern-0.1em\Mod}%
\nc{\otoo}[1]{\overset{#1}{\,\too\,}}
\nc{\John}[1]{{[\color{Green}#1]}}
\nc{\Paul}[1]{{[\color{Blue}#1]}}
\nc{\too}{\mathop{\longrightarrow}\limits}
\nc{\threestars}{\medbreak\begin{center}*\ *\ *\end{center}}
\nc{\calO}{\mathcal{O}}

\nc{\bbA}{\mathbb{A}}
\nc{\bbD}{\mathbb{D}}
\nc{\bbN}{\mathbb{N}}
\nc{\bbZ}{\mathbb{Z}}

\title[Affine space over a triangulated category]{Affine space over triangulated categories:\\ A further invitation to Grothendieck derivators}
\author{Paul Balmer}
\address{Paul Balmer, Mathematics Department, UCLA, Los Angeles, CA 90095-1555, USA}
\email{balmer@math.ucla.edu}
\urladdr{http://www.math.ucla.edu/$\sim$balmer}
\thanks{Supported by NSF grant~DMS-1303073 and Research Award of the Humboldt Foundation.}

\author{John Zhang}
\address{John Zhang, Mathematics Department, UCLA, Los Angeles, CA 90095-1555, USA}
\email{jmzhang@math.ucla.edu}
\urladdr{http://www.math.ucla.edu/$\sim$jmzhang}

\date{\today} 

\begin{document}

\begin{abstract}
We propose a construction of affine space (or ``polynomial rings'') over a triangulated category, in the context of stable derivators.
\end{abstract}

\maketitle

\smallbreak
\ \hfill \textit{Dedicated to Chuck Weibel on the occasion of his 65th birthday}
\medbreak


\subsection*{Introduction}\ \smallbreak

Among their various incarnations throughout mathematics, triangulated categories play an important role in algebraic geometry as derived categories of schemes, beginning with Grothendieck's foundational work, see~\cite{Grothendieck60} or~\cite{Hartshorne66}.

In fact, the Grothendieck-Verdier formalism by means of distinguished triangles~\cite{Verdier96,Neeman01} can be considered as the lightest axiomatization of stable homotopy theory available on the market. Heavier approaches can involve model categories, dg-enrichments, $\infty$-categories, etc. The simplicity of triangulated categories presents advantages in terms of versatility and applications. For instance we expect most readers, way beyond algebraic geometers, to have some acquaintance with triangulated categories. On the other hand, this light axiomatic can also be a weakness, notably when it comes to producing new triangulated categories out of old ones. For example, the triangular structure is sufficient to describe Zariski localization~\cite{ThomasonTrobaugh90,Neeman92b}, or even \'etale extensions~\cite{Balmer14pp}, but there is no known construction, purely in terms of triangulated categories, that would describe the derived category of affine space~$\bbA^n_X$ out of the derived category of~$X$. This is the problem we want to address.

\smallbreak
The goal of this short note is to highlight a straightforward and elementary way to construct affine space if one uses the formalism of \emph{derivators}. (See Theorem~\ref{thm:main}.)
\smallbreak

Currently, derivators are nowhere as well-known as triangulated categories and it is fair to assume that very few of our readers really know what a derivator is. The non-masochistic reader will even appreciate that the derivator solution to the above affine space problem is \emph{essentially trivial}, a further indication of the beauty of this theory. We hope that our observation will serve as a modest additional incentive for some readers to use the axiomatic of Grothendieck derivators~\cite{Grothendieck91}, or the closely related approach of Heller~\cite{Heller88}.

Our observation also opens the door to an $\bbA^1$-homotopic investigation of derivators, which will be the subject of further work.

\subsection*{A glimpse at derivators}\ \smallbreak

The idea behind derivators is simple. Suppose that your triangulated category of interest, say~$\cat D$, is the homotopy category $\cat D=\Ho(\cat M)$ of some ``model''~$\cat M$, by which we mean that $\cat D$ can be obtained from~$\cat M$ by inverting a reasonable class of maps, called \emph{weak-equivalences}. For instance $\cat M$ could be the category of complexes of $R$-modules, with quasi-isomorphisms as weak-equivalences. Then for any small category~$I$, the category of functors $\cat M^{I}$ from~$I$ to~$\cat M$ (\aka $I$-shaped diagrams in~$\cat M$) should also yield a homotopy category by inverting objectwise weak-equivalences:
\begin{equation}
\label{eq:Ho}%
\bbD(I):=\Ho(\cat M^I)\,.
\end{equation}
This approach requires some care but if the word ``model'' is given a precise meaning, it works, see~\cite[p.2-3]{Heller88} or~\cite[p.2]{Groth13}. Now, if $I=e$ is the final category with one object and one identity morphism, then $\bbD(e)=\cat D$ is the category we started from, since $\cat M^{e}=\cat M$. For another example, if $I=[1]=\{\bullet\to \bullet\}$ is the category with two objects and one non-identity morphism, then $\bbD([1])=\Ho(\Arr(\cat M))$ is the homotopy category of arrows (refining the category $\Arr(\Ho(\cat M))$ of arrows in~$\cat D$, which is never triangulated unless $\cat D=0$); and similarly for other categories~$I$.

This collection of homotopy categories $\bbD(I)=\Ho(\cat M^I)$, indexed over all small categories~$I$, varies functorially in~$I$, because every functor $u:I\to J$ induces a functor $u^*:\bbD(J)\to \bbD(I)$ via precomposition with~$u$. Also, every natural transformation $\alpha:u\rightarrow v$ between functors $u,v:I\rightarrow J$ induces a natural transformation~$\alpha^*:u^*\rightarrow v^*$; so in technical terms, $\bbD$ is a \emph{strict 2-functor}
\[
\bbD:\Cat^{\op,\textrm{-}} \too \CAT
\]
from the 2-category of small categories~$\Cat$ to the (big) 2-category~$\CAT$ of all categories (the ``op" indicates that $\bbD$ reverses direction of 1-morphisms, $u\mapsto u^*$).

Independently of the existence of a model~$\cat M$ as above, an arbitrary strict 2-functor $\bbD:\Cat^{\op,\textrm{-}}\too \CAT$ is called a \emph{pre-derivator} and the category~$\bbD(e)$ is called the \emph{base} of the prederivator~$\bbD$.

Actual \emph{derivators} are pre-derivators satisfying a few axioms giving them some spine. In addition to~\cite{Grothendieck91}, contemporary literature offers several good sources to learn the precise axioms, see for instance the detailed treatment in Groth~\cite{Groth13} or a shorter account in Cisinski-Neeman~\cite[\S\,1]{CisinskiNeeman08}. We shall not repeat the axioms in this short note, since our focus is more towards \emph{constructing} new derivators out of old ones, but we would like to point out that these axioms are relatively simple. Moreover, if $\bbD$ is \emph{stable}, then its values $\bbD(I)$ have natural triangulations that are moreover compatible in the sense that the functors $u^*$ are triangulated functors.

Of course, we have a derivator extending the derived category of a ring:
\begin{theorem}[Grothendieck]
\label{thm:der}%
Let $R$ be a ring and consider the category of chain complexes $\cat M=\Ch(R\MMod)$ with weak-equivalences the quasi-isomorphisms. Then formula~\eqref{eq:Ho} defines a derivator that we shall denote by~$\DER(R\MMod)$. Its base is the usual derived category~$\Der(R\MMod)$.
\end{theorem}

References about the above are given in Remark~\ref{rem:der} below.

\subsection*{Statement of the result}\
\smallbreak

For our purpose, we only need one basic technique from derivator theory:

\begin{definition}
\label{def:shift}%
Given a fixed small category~$L$, we can \emph{shift} a derivator~$\bbD$ by~$L$ and obtain a new derivator~$\bbD^L$ defined by~$\bbD^L(I)=\bbD(L\times I)$.
\end{definition}

More precisely, $\bbD^L$ is the composite 2-functor $\Cat^{\op,\textrm{-}} \ \otoo{L\times-} \ \Cat^{\op,\textrm{-}} \otoo{\bbD} \CAT$. The fact that this prederivator~$\bbD^L$ remains a derivator is not completely trivial but it is a useful part of the elementary theory, see~\cite[Theorem 1.31]{Groth13}.

Now, we want to follow the intuition that a module over a polynomial ring~$R[T]$ is just an $R$-module together with a chosen $R$-linear endomorphism corresponding to the action of~$T$. Of course, once $T$ acts on an $R$-module then $T^n$ also acts for all~$n\in\bbN$. Consider, therefore, the \emph{loop} category
\begin{equation}
\label{eq:loop}%
\Loop=
\xymatrix{\bullet \ar@(ur,dr)[]^-{\bbN}}
\end{equation}
which has a single object with endomorphism monoid $\bbN=(\{0,1,2,3,\ldots\},+)$.

\begin{theorem}
\label{thm:main}%
Let $R$ be a ring and let $\bbD=\DER(R\MMod)$ be the derivator extending the derived category of~$R$ (Thm.\,\ref{thm:der}). Then there exists a canonical and natural isomorphism of derivators between $\DER(R[T]\MMod)$ and the derivator $\bbD^{\Loop}$ obtained by shifting~$\DER(R\MMod)$ by the loop category~$\Loop$, which is moreover a ``strong equivalence.'' Of course, there is also a similar canonical and natural isomorphism between $\DER(R[T_1,\ldots,T_n]\MMod)$ and $\bbD^{\Loop^n}$.
\end{theorem}

The word ``strong'' in the statement refers to the notion of ``strong equivalence'' of derivators due to Muro and Raptis~\cite[Def.\,3.3.2]{MuroRaptis14pp}. Evaluation at the terminal category~$I=e$, that is, on the base categories, explains the title:

\begin{corollary}
If $\bbD=\DER(R\MMod)$ is the derivator extending the derived category of~$R$, then $\bbD(\Loop^n)$ is the derived category of $R[T_1,\ldots,T_n]$ for any~$n\geq1$. \qed
\end{corollary}

\smallbreak
\subsection*{The proof}\
\smallbreak

Let $R$ be a ring, $\Loop=\xymatrix{\bullet \ar@(ur,dr)[]^-{\bbN}}$ be the loop category~\eqref{eq:loop}. The starting point is the well-known canonical isomorphism of categories
\begin{equation}
\label{eq:polynomial}%
(R\MMod)^{\Loop}\cong R[T]\MMod.
\end{equation}
It maps a functor $A:\Loop\to R\MMod$ to the $R[T]$-module which is $A(\bullet)$ as $R$-module, on which $T$ acts as $A(1:\bullet\to \bullet)\,:\ A(\bullet)\to A(\bullet)$; and it maps a natural transformation $A\to A'$ to the associated morphism at the single object~$\bullet$. In the other direction, one associates to an $R[T]$-module~$M$ the functor which takes the value~$M$ (or rather its restriction as an $R$-module) at the unique object~$\bullet$ and maps any morphism $n:\bullet\to \bullet$ in~$\Loop$ to the morphism $(T^n\cdot-):M\to M$; the natural transformation associated to a morphism $M\to M'$ is the obvious one.

The plan is now to pass this isomorphism~\eqref{eq:polynomial} through the construction of the associated derivators, ``taking out'' the exponent~$\Loop$ as we go. From now on, the shifting category~$\Loop$ can be replaced by any small category~$L$.

\begin{lemma}
\label{lem:Ch}%
Let $L$ be a small category and $\cat A$ be an abelian category. There is a canonical isomorphism of categories
\[
(\Ch(\cat A))^L\simeq \Ch(\cat A^L)
\]
If we define weak-equivalences in~$(\Ch(\cat A))^L$ to be quasi-isomorphisms objectwise on~$L$, and weak-equivalences in~$\Ch(\cat A^L)$ are quasi-isomorphisms for the abelian category~$\cat A^L$, then this canonical isomorphism preserves weak-equivalences.
\end{lemma}

\begin{proof}
Let $\bbZ_{\scriptscriptstyle\leq}$ be the category associated to the poset~$\bbZ$, that is, with one object for each integer and one morphism $n\to m$ whenever $n\leq m$. Both $(\Ch(\cat A))^L$ and $\Ch(\cat A^L)$ can be visualized as subcategories of $(L\times \bbZ_{\scriptscriptstyle\leq})$-shaped diagrams in~$\cat A$ where the composition of two consecutive morphisms in the $ \bbZ_{\scriptscriptstyle\leq}$-direction is 0. Since kernels and cokernels in~$\cat A^L$ are computed objectwise on~$L$, weak-equivalences on both sides are morphisms which, objectwise on~$L$, are quasi-isomorphisms in~$\cat A$.
\end{proof}

\begin{definition}
\label{def:M,W}%
Let $\cat M$ be a category and $\cat W$ be a class of morphisms in~$\cat M$ that we call ``weak-equivalences.'' For example, $\cat M$ might be a model category or Waldhausen category and $\cat W$ could be the weak equivalences determined by the model category or Waldhausen structure. Given a small category~$I$, we equip $\cat M^I$ with the objectwise weak-equivalences~$\cat W^I$, as above. We say that the pair $(\cat M,\cat W)$ \emph{induces a derivator} if, for every small category~$I$, the localization $\cat M^I[(\cat W^I)\inv]$ of~$\cat M^I$ with respect to~$\cat W^I$ is a category and more importantly if the induced prederivator $I\mapsto \cat M^I[(\cat W^I)\inv]$ is a derivator. In that case we denote the derivator induced by $(\cat M,\cat W)$ as
\[
\HO(\cat M,\cat W)\,:\Cat^{\op,\textrm{-}}\too \CAT\,.
\]
\end{definition}

\begin{remark}
\label{rem:der}%
In this language, Theorem~\ref{thm:der} says that $(\cat M=\Ch(R\MMod), \cat W=\{\textrm{q-isos}\})$ induces a derivator. One way to demonstrate this is to put a Quillen model structure on $\Ch(R\MMod)$ and apply Cisinski's result~\cite[Theorem~1]{Cisinski03}, which says that every model category induces a derivator via formula~\eqref{eq:Ho}. In fact, as long as the weak-equivalences are fixed, one can play with different model structures on~$\cat M^I$ to ensure in turn that for any functor $u:I\rightarrow J$, $u^*:\cat M^J\to \cat M^I$ is a left or a right Quillen functor, thus obtaining the right and left adjoints to $u^*$ (required as part of the definition of derivator) as derived functors themselves. Alternatively, every Grothendieck abelian category $\cat A$ also induces a derivator via formula~\eqref{eq:Ho}, using the pair $(\cat M=\Ch(\cat A),\cat W=\{\textrm{q-isos}\})$; see~\cite[Chap.III, p.10]{Grothendieck91} where the case $\cat A=R\MMod$ is explicitly given.
\end{remark}

\begin{proposition}
\label{prop:M^L}%
Let $\cat M$ be a category with a class~$\cat W$ of weak-equivalences and let $L$ be a small category. Suppose that $(\cat M,\cat W)$ induces a derivator in the sense of Definition~\ref{def:M,W} and let $\bbD=\HO(\cat M,\cat W)$. Then $(\cat M^L,\cat W^L)$ induces a derivator, and we have a natural isomorphism of derivators $\HO(\cat M^L,\cat W^L)\cong \bbD^L$, which is moreover a strong equivalence.
\end{proposition}

\begin{proof}
We have a natural isomorphism $(\cat M^L)^I\cong\cat M^{L\times I}$ by the categorical exponential law. Moreover the respective weak-equivalences $(\cat W^L)^I$ and $\cat W^{L\times I}$ in the two categories are preserved by this isomorphism. Hence the condition for $(\cat M^L,\cat W^L)$ to induce a derivator follows from the corresponding hypothesis on~$(\cat M,\cat W)$.

For a given $I$, the values of the derivators $\HO(\cat M^L, \cat W^L)$ and $\bbD^L$ are
\[
\HO(\cat M^L,\cat W^L)(I)=(\cat M^L)^I[((\cat W^L)^I)^{-1}]
\quad\textrm{and}\quad
\bbD^L(I)=\cat M^{L\times I}[(\cat W^{L\times I})^{-1}]
\]
respectively. These categories are isomorphic by the above discussion, hence the derivators are isomorphic. Moreover, these isomorphisms assemble to strict morphisms of derivators and therefore define strongly equivalent derivators in the sense of ~\cite{MuroRaptis14pp}, as the specialized reader will directly verify.
\end{proof}

\begin{proof}[Proof of Theorem~\ref{thm:main}]
We have $(R\MMod)^{\Loop}\cong R[T]\MMod$ by~\eqref{eq:polynomial}, from which we get the isomorphism of derivator-defining (Def.\,\ref{def:M,W}) pairs $\big(\Ch(R[T]\MMod)\,,\,\{\textrm{q.isos}\}\big)\cong \big(\Ch(R\MMod)\,,\,\{\textrm{q.isos}\}\big)^{\Loop}$ by Lemma~\ref{lem:Ch}. Then Proposition~\ref{prop:M^L} gives us the statement about the associated derivators. The case of~$R[T_1,\ldots,T_n]$ follows by induction.
\end{proof}

\begin{remark}
We do not use that $R$ is commutative and the result probably holds in broader generality. We leave such variations to the interested reader and only quickly outline the scheme case, at the referee's suggestion.
\end{remark}

\smallbreak
\subsection*{The scheme case}\
\smallbreak

Let $X$ be a quasi-compact and \emph{separated} scheme. We can associate a derivator to $X$ by $\bbD_{X}(I)=D(\Qcoh(X)^I)$, where $\Qcoh(X)$ denotes the category of quasi-coherent $\calO_X$-modules. As $\Qcoh(X)$ is a Grothendieck abelian category, this is a stable derivator. If $X$ was not separated but only quasi-separated, the appropriate ``derived category'' of $X$ would be the derived category of $\calO_X$-modules with quasi-coherent homology. Under our assumption that $X$ is separated, this coincides with the above. This mild restriction is irrelevant for the point we want to make.

\begin{proposition}
We have an isomorphism of categories $\Qcoh(\bbA^1_{X})\cong \Qcoh(X)^{\Loop}$. 
\end{proposition}

\begin{proof}
This is a global analogue of the affine case~\eqref{eq:polynomial}, from which it follows by Zariski descent. Indeed, if $\{U_i=\Spec A_i\}$ is an open cover of~$X$, then $\{\bbA^1_{U_i}=\Spec A_i[t]\}$ form an open cover of $\bbA^1_{X}$. The obvious functor $\Qcoh(\bbA^1_X)\to \Qcoh(X)^{\Loop}$ restricts to the canonical isomorphism~\eqref{eq:polynomial} on each open~$U_i$, on each double intersection (for gluing data), and on each triple intersection (for cocycle condition). The result follows by direct inspection.
\end{proof}

The category of quasi-coherent sheaves on $X$ is Grothendieck abelian, so we may apply Lemma~\ref{lem:Ch} and Proposition~\ref{prop:M^L} as in the proof of Theorem~\ref{thm:main} to obtain the isomorphism of derivators~$(\bbD_{X})^{\Loop}\cong\bbD_{\bbA^1_{X}}$. On base categories, this isomorphism describes the derived category of~$\bbA^1_X$ as $\bbD_{X}(\Loop)$.




\end{document}